\documentclass[11pt,twoside,reqno]{amsart}

\usepackage{microtype}
\usepackage{version}         

\usepackage{color}

\usepackage[OT1]{fontenc}
\usepackage{type1cm}
\usepackage{amssymb}

\usepackage[left=2.3cm,top=3cm,right=2.3cm]{geometry}

\numberwithin{equation}{section}

\theoremstyle{plain}
\newtheorem{theorem}{Theorem}[section]

\newtheorem{corollary}[theorem]{Corollary}

\newtheorem{proposition}[theorem]{Proposition}

\theoremstyle{remark}

\newtheorem{example}[theorem]{Example}

\theoremstyle{definition}

\newtheorem{question}[theorem]{Question}

\newsavebox{\framedbox}

\newcommand{\BB}{\mathcal{B}}

\newcommand{\HH}{\mathcal{H}}

\newcommand{\QQ}{\mathcal{Q}}

\newcommand{\R}{\mathbb{R}}

\newcommand{\Z}{\mathbb{Z}}
\newcommand{\N}{\mathbb{N}}

\newcommand{\iii}{\mathtt{i}}
\newcommand{\jjj}{\mathtt{j}}
\newcommand{\kkk}{\mathtt{k}}
\newcommand{\eps}{\varepsilon}
\newcommand{\fii}{\varphi}
\newcommand{\roo}{\varrho}
\newcommand{\ualpha}{\overline{\alpha}}
\newcommand{\lalpha}{\underline{\alpha}}

\newcommand{\Lip}{\operatorname{Lip}}

\DeclareMathOperator{\dist}{dist}
\DeclareMathOperator{\diam}{diam}

\DeclareMathOperator{\dima}{\overline{dim}_A}
\DeclareMathOperator{\diml}{\underline{dim}_A}

\DeclareMathOperator{\udimreg}{\overline{dim}_{reg}}
\DeclareMathOperator{\ldimreg}{\underline{dim}_{reg}}

\DeclareMathOperator{\dimh}{dim_H}

\DeclareMathOperator{\udimm}{\overline{dim}_M}
\DeclareMathOperator{\ldimm}{\underline{dim}_M}

\DeclareMathOperator{\dimp}{dim_p}

\begin{document}

\title
{Measures with predetermined regularity and inhomogeneous self-similar sets}

\author{Antti K\"aenm\"aki}
\author{Juha Lehrb\"ack}
\address{Department of Mathematics and Statistics \\
         P.O.\ Box 35 (MaD) \\
         FI-40014 University of Jyv\"askyl\"a\\
         Finland}
\email{antti.kaenmaki@jyu.fi}
\email{juha.lehrback@jyu.fi}

\thanks{JL has been supported in part by the Academy of Finland (project \#252108)}
\date{\today}

\begin{abstract}
    We show that if $X$ is a uniformly perfect complete metric space satisfying the finite doubling property, then there exists a fully supported measure with lower regularity dimension as close to the lower dimension of $X$ as we wish. Furthermore, we show that, under the condensation open set condition, the lower dimension of an inhomogeneous self-similar set $E_C$ coincides with the lower dimension of the condensation set $C$, while the Assouad dimension of $E_C$ is the maximum of the Assouad dimensions of the corresponding self-similar set $E$ and the condensation set $C$. If the Assouad dimension of $C$ is strictly smaller than the Assouad dimension of $E$, then the upper regularity dimension of any measure supported on $E_C$ is strictly larger than the Assouad dimension of $E_C$. Surprisingly, the corresponding statement for the lower regularity dimension fails.
\end{abstract}

\maketitle

\section{Introduction}

It is a well-known fact that every complete doubling metric space $X$ carries a doubling measure; see \cite{LuukkainenSaksman1998}. Even more is true: for each $s>\dima(X)$ there exists a doubling measure $\mu$ supported on $X$ such that $\udimreg(\mu)<s$. Here  $\dima(X)$ is the (upper) Assouad dimension of $X$ and $\udimreg(\mu)$ is the upper regularity dimension of $\mu$; see \S \ref{sect:metric} for the definitions. Recall that $X$ is doubling if and only if $\dima(X)<\infty$, and that the bound $\dima(X)\le \udimreg(\mu)$ is always valid for a doubling measure $\mu$ supported on $X$.
In the case of a compact space $X$, the existence of such a measure was proven by Vol'berg and Konyagin~\cite[Theorem~1]{VolbergKonyagin1987}; see also~\cite[Section~13]{Heinonen2001}. They also gave an example of a compact metric space $X$ where $\dima(X) < \udimreg(\mu)$ for all doubling measures $\mu$ supported on $X$; see~\cite[Theorem~4]{VolbergKonyagin1987}. In~\cite{LuukkainenSaksman1998}, Luukkainen and Saksman generalized the existence result to complete spaces using the compact case and a limiting argument, and in~\cite{KaenmakiRajalaSuomala2012a}, an elegant direct proof in the complete case was given using a suitable ``nested cube structure'' of doubling metric spaces.  

On the other hand, in~\cite{BylundGudayol2000}, Bylund and Gudayol considered the ``dual'' question for the above results. A particular consequence of their main result~\cite[Theorem~9]{BylundGudayol2000} is that if $X$ is a complete doubling (pseudo-)metric space, then for each $0\le t<\diml(X)$ there exists a doubling measure $\mu$ supported on $X$ such that $\ldimreg(\mu)>t$. Here $\diml(X)$ is the lower (Assouad) dimension of $X$ and $\ldimreg(\mu)$ is the lower regularity dimension of $\mu$. As above, the bound $\diml(X)\ge \ldimreg(\mu)$ is always valid for such measures $\mu$. Again, the results in~\cite{BylundGudayol2000} are first established for compact spaces, and a limiting argument as in~\cite{LuukkainenSaksman1998} takes care of the extension to complete spaces. We remark that the result \cite[Theorem~9]{BylundGudayol2000} in fact takes into account both upper and lower regularities and states that if $X$ is a complete doubling (pseudo-)metric space and $0\le t<\diml(X)\le \dima(X)<s<\infty$, then there exists a measure $\mu$ supported on $X$ such that $t<\ldimreg(\mu)\le \udimreg(\mu)<s$. Here the middle inequality is a triviality.

In this paper, we extend and clarify several aspects related to the above results. First of all, in Theorem~\ref{thm:finite-condition}, we generalize the ``lower'' part of the Bylund and Gudayol result. Instead of the doubling condition, we assume a weaker finite doubling property, and  
show that then for each $0\le t<\diml(X)$ there exists a measure $\mu$ which is supported on the (uniformly perfect) metric space $X$ and satisfies $\ldimreg(\mu)>t$; recall that $X$ is uniformly perfect if and only if $\diml(X)>0$. The proof of Theorem~\ref{thm:finite-condition} is rather transparent and works immediately in the the non-compact case in the spirit of~\cite{KaenmakiRajalaSuomala2012a}. 
If $X$ is doubling, then our construction can be modified so that also the resulting measure $\mu$ is doubling, and hence we recover the ``lower'' part of the conclusion of Bylund and Gudayol; see Theorem~\ref{thm:doubling}. As a consequence of Theorem~\ref{thm:finite-condition}, we see that if $X$ satisfies the finite doubling property, then $X$ is uniformly perfect if and only if $X$ supports a reverse-doubling measure.

Our second point of interest lies in the general phenomena behind the example of Vol'berg and Konyagin~\cite[Theorem~4]{VolbergKonyagin1987}, where the space $X$ is a simple special case of an inhomogeneous self-similar set. In Theorem~\ref{thm:upper_reg}, we show that in fact a similar conclusion holds for a large class of inhomogeneous self-similar sets with small enough condensation sets. After this a natural question is whether a corresponding result is true also for lower regularities and lower dimensions of inhomogeneous self-similar sets. Despite natural ``dualities'' concerning many other questions related to upper and lower regularities and Assouad and lower dimensions (cf.\ e.g.~\cite{KaenmakiLehrbackVuorinen2013}), the answer here is negative, as we show in Proposition~\ref{prop:non-VK}. The existence of spaces with the property that $\ldimreg(\mu)<\diml(X)$ for all measures $\mu$ supported on $X$ is thus left as an open question.

Since the above considerations rely on sufficient knowledge of the Assouad and lower dimensions of inhomogeneous self-similar sets, which (to our best knowledge) are not available in the literature, we need to establish such results which of course have also independent interest. If $E_C$ is an inhomogeneous self-similar set, then the formula
\begin{equation*}
  \dim(E_C) = \max\{ \dim(E), \dim(C) \}
\end{equation*}
holds for the Hausdorff and packing dimensions without any separation conditions. By assuming the open set condition, Fraser \cite{Fraser2012} proved the formula for the upper Minkowski dimension. Baker, Fraser, and M\'ath\'e \cite{BakerFraserMathe2017} later observed that it may fail without the open set condition. By posing an additional separation condition on the condensation set, we prove the above formula for the Assouad dimension in Theorem \ref{thm:dim of inhomog}. For the lower dimension we show in Theorem \ref{thm:dim of inhomog2} that $\diml(E_C) = \diml(C)$, also under this additional separation condition. Neither of these results holds if we only assume the open set condition; see Example~\ref{ex:no_separation}.

\section{Setting}\label{sect:metric}

Let $X=(X,d)$ be a metric space. The closed ball with center $x\in X$ and radius $r>0$ is $B(x,r)=\{y\in X : d(x,y)\le r\}$. For notational convenience, we keep to the convention that each ball $B\subset X$ comes with a fixed center and radius. This makes it possible to use notation such as $2B=B(x,2r)$ without explicit reference to the center and the radius of the ball $B=B(x,r)$. For convenience, we also make the general assumption that the space $X$ contains at least two points.

We say that $X$ satisfies the \emph{finite doubling property} if any ball $B(x,2r) \subset X$ can be covered by finitely many balls of radius $r$. Furthermore, $X$ is \emph{doubling} if the number of balls above is uniformly bounded by $N\in\N$. 
Iteration of the doubling condition shows that if $X$ is doubling, then there are constants $C \ge 1$ and $0\le s\le \log_2 N$ such that each ball $B(x,R)$ can be covered by at most $C(r/R)^{-s}$ balls of radius $r$ for all $0<r<R<\diam(X)$. The infimum of exponents $s$ for which this holds (for some constant $C=C(s,X)$) is called the \emph{Assouad} (or \emph{upper Assouad}) \emph{dimension} of $X$ and is denoted by $\dima(X)$. Hence $X$ is doubling if and only if $\dima(X)<\infty$.
We refer to Luukkainen \cite{Luukkainen1998} for the basic properties and a historical account on the Assouad dimension. 

Conversely to the above definition, we may also consider all $t \ge 0$ for which there is a constant $c>0$ so that if $0<r<R<\diam(X)$, then for every $x \in X$ at least $c(r/R)^{-t}$ balls of radius $r$ are needed to cover $B(x,R)$. The supremum of all such $t$ is called the \emph{lower} (or \emph{lower Assouad}) \emph{dimension} of $X$ and is denoted by $\diml(X)$. 
Considering the restriction metric, the definitions of Assouad and lower dimensions extend to all subsets $E\subset X$, although in the case of one-point sets $E=\{x_0\}$ we remove the restriction $R<\diam(E)$.  

A metric space $X$ is \emph{uniformly perfect} if there exists a constant $C \ge 1$ so that for every $x \in X$ and $r>0$ we have $B(x,r) \setminus B(x,r/C) \ne \emptyset$ whenever $X \setminus B(x,r) \ne \emptyset$. In~\cite[Lemma~2.1]{KaenmakiLehrbackVuorinen2013}, it was shown that a metric space $X$ with $\# X\ge 2$ is uniformly perfect if and only if $\diml(X) > 0$.

By a \emph{measure} we exclusively refer to a nontrivial Borel regular outer measure for which bounded sets have finite measure.
We say that a measure $\mu$ on $X$ is \emph{doubling} if there is a constant $C \ge 1$ such that
$0 < \mu(2B) \le C\mu(B)$ 
for all closed balls $B$ of $X$.
The existence of a doubling measure yields an upper bound for the Assouad dimension of $X$. Indeed, 
if $\mu$ is doubling, then there exist $s>0$ and $c>0$ such that
\begin{equation*} \label{eq:upper_reg}
  \frac{\mu(B(x,r))}{\mu(B(x,R))}\ge c\Bigl(\frac rR\Bigr)^s
\end{equation*}
for all $x\in X$ and $0<r<R<\diam(X)$. The infimum of such admissible exponents $s$ is called the \emph{upper regularity dimension} of $\mu$ and is denoted by $\udimreg(\mu)$. A simple volume argument implies that $\dima(X) \le \udimreg(\mu)$ whenever $\mu$ is a doubling measure on $X$. 

Conversely, 
the supremum of $t\ge 0$ for which there exists a constant $C \ge 1$ such that
\begin{equation}\label{eq:lower_reg}
  \frac{\mu(B(x,r))}{\mu(B(x,R))}\leq C\Bigl(\frac rR\Bigr)^t,
\end{equation}
whenever $x\in X$ and $0<r<R<\diam(X)$, is called the \emph{lower regularity dimension} of $\mu$ and is denoted by $\ldimreg(\mu)$. It follows directly from~\eqref{eq:lower_reg} that $\ldimreg(\mu) \le \diml(X)$. From~\cite[Lemma~3.1]{KaenmakiLehrbackVuorinen2013} it follows that $\ldimreg(\mu)>0$ for a doubling measure $\mu$ in a uniformly perfect space $X$. 
Notice also that the property $\ldimreg(\mu)>0$ is equivalent to the \emph{reverse-doubling} condition that there exist constants $c>1$ and $\tau>1$ such that
$\mu(B(x,\tau r)) \ge c\mu(B(x,r))$ 
whenever $x\in X$ and $0<r\le \diam X/(2\tau)$.

A measure $\mu$ is \emph{$s$-regular} (for $s > 0$) if there is a constant $C \ge 1$ such that
\[
  C^{-1}r^s \le \mu(B(x,r)) \le Cr^s
\]
for all $x \in X$ and every $0 < r < \diam(X)$. 
It is immediate that if $X$ is bounded, then a measure $\mu$ is $s$-regular if and only if $s = \ldimreg(\mu) = \udimreg(\mu)$. We say that $X$ is \emph{$s$-regular} if it carries an $s$-regular measure $\mu$. A subset $E \subset X$ is \emph{$s$-regular} if it is an $s$-regular space in the relative metric.

We follow the convention that letters $C$ and $c$ are sometimes used to denote positive constants whose exact value is not important and may be different at each occurrence. If $a$ and $b$ are quantities such that $a\le Cb$, then we may write $a\lesssim b$, and if $a\lesssim b\lesssim a$, then $a\approx b$.

\section{Measures with lower regularity close to lower dimension}

In this section, we construct a measure supported on the whole space so that its lower regularity dimension is as close to the lower dimension of the space as we wish.
We begin by recalling the following construction of ``dyadic cubes'' from \cite[Theorem~2.1]{KaenmakiRajalaSuomala2012a}.

\begin{proposition} \label{prop:dyadic}
  If $X$ is a metric space satisfying the finite doubling property and $0<\roo<\frac13$, then there exists a collection $\{ Q_{k,j} : k \in \Z \text{ and } j \in N_k \subset \N \}$ of Borel sets having the following properties:
  \begin{enumerate}
    \item\label{D1} $X = \bigcup_{j \in N_k} Q_{k,j}$ for every $k \in \Z$,
    \item\label{D2} $Q_{m,i} \cap Q_{k,j} = \emptyset$ or $Q_{m,i} \subset Q_{k,j}$ whenever $k,m \in \Z$, $m \ge k$, $i \in N_m$, and $j \in N_k$,
    \item\label{D3} for every $k\in \Z$ and $j \in N_k$ there exists a point $x_{k,j} \in X$ so that
    \begin{equation*}
      B(x_{k,j},c \roo^k) \subset Q_{k,j} \subset B(x_{k,j},C\roo^k)
    \end{equation*}
    where $c = \tfrac12-\tfrac{\roo}{1-\roo}$ and $C = \tfrac{1}{1-\roo}$,
    \item\label{D4} there exists a point $x_0\in X$ so that for every $k\in \Z$ there is $j\in N_k$ so that $B(x_0,cr^k)\subset Q_{k,j}$.
  \end{enumerate}
\end{proposition}

Relying on the existence of such cubes we are able to construct the desired measure.

\begin{theorem} \label{thm:finite-condition}
  Suppose that $X$ is a uniformly perfect complete metric space satisfying the finite doubling property. Then for each $\eps>0$ there exists a measure $\mu$ supported on $X$ such that $\ldimreg(\mu) > \diml(X)-\eps$.
\end{theorem}

\begin{proof}
  Fix $0<t<s<\diml(X)$. To show the claim it suffices to construct a measure $\mu$ supported on $X$ such that $\mu$ satisfies \eqref{eq:lower_reg} with $t$. The idea of the construction is to distribute mass along the ``dyadic cubes'' of Proposition~\ref{prop:dyadic} so that the center part of the cube always gets the biggest portion in a suitable way.
  
  Relying on $s<\diml(X)$, let $c_0>0$ be such that at least $c_0(r/R)^{-s}$ balls of radius $r>0$ are needed to cover a ball of radius $R>r$. Fix $\lambda = 1/8$. Given $0<\roo<\tfrac13$, let $c = \frac12-\frac{\roo}{1-\roo}$ and $C = \frac{1}{1-\roo}$. Note that if $0<\roo<\tfrac15$, then $\tfrac14 < c < \tfrac12$ and $1 < C < \tfrac 5 4 < 2$, and so in particular $C\roo < 1$. 
  Moreover, if $\roo < \sqrt{2}-\tfrac 5 4 < \sqrt{C^2+1}-C$, then
  \begin{equation} \label{eq:est3}
    \roo^2 + 2C\roo < 1,
  \end{equation} 
  and if $\roo < 1/64$, then
  \begin{equation} \label{eq:est2}
    (2C+4)\roo + \lambda < 8\roo + \tfrac18 < \tfrac14 < c.
  \end{equation} 
  Furthermore, if $\roo < (c_02^{-s}\lambda^s)^{1/(s-t)}$, then
  \begin{equation} \label{eq:est1}
    c_0 C^{-s}\lambda^s \roo^{-s} \ge c_0 2^{-s}\lambda^s \roo^{-s} > \roo^{-t}.
  \end{equation}
  Thus, we may fix $\roo>0$ so that all of estimates~\eqref{eq:est3}, \eqref{eq:est2}, and~\eqref{eq:est1} hold. For this $\roo$, let $\QQ = \{ Q_{k,j} : k \in \Z \text{ and } j \in N_k \}$ be as in Proposition~\ref{prop:dyadic}.

  Fix $k \in \Z$ and $j \in N_k$. Let $x \in X$ and $r>0$ and define
  \begin{align*}
    N_{k+1}^j &= \{ i \in N_{k+1} : Q_{k+1,i} \subset Q_{k,j} \}, \\
    Q_k(x,r) &= \{ j \in N_k : Q_{k,j} \cap B(x,r) \ne \emptyset \}.
  \end{align*}
  By Proposition~\ref{prop:dyadic}(3), the assumption $s < \diml(X)$ and~\eqref{eq:est1} imply
  \begin{equation*}
  \# Q_{k+1}(x_{k,j},\lambda\roo^k) \ge c_0\biggl( \frac{C\roo^{k+1}}{\lambda\roo^k} \biggr)^{-s} \ge c_0 C^{-s}\lambda^s \roo^{-s} \ge \roo^{-t}. 
  \end{equation*}
  Since $X$ satisfies the finite doubling property, we see in a similar way that $\# N_{k+1}^j < \infty$ for all $k \in \Z$ and $j \in N_k$.
  
  Let $K = \lceil \roo^{-t} \rceil$ and choose $N_{k+1}^j(\lambda) \subset Q_{k+1}(x_{k,j},\lambda\roo^k)$ such that $\# N_{k+1}^j(\lambda) = K$. Observe that since $\lambda < c$ we have $N_{k+1}^j(\lambda) \subset N_{k+1}^j$. Let $\eps_{k,j} = ((K+1)\# N_{k+1}^j)^{-1}$ and define a set function $\mu_0 \colon \QQ \to [0,\infty)$ as follows: First fix $i_0 \in N_0$ and set $\mu_0(Q_{0,i_0})=1$ and then require that
  \begin{equation} \label{eq:mu0-def}
    \mu_0(Q_{k+1,i}) =
    \begin{cases}
      \bigl(\frac{1}{K+1}+\eps_{k,j}\bigr) \mu_0(Q_{k,j}), &\text{if } i \in N_{k+1}^j(\lambda) \\
      \eps_{k,j} \mu_0(Q_{k,j}), &\text{if } i \in N_{k+1}^j \setminus N_{k+1}^j(\lambda)
    \end{cases}
  \end{equation}
  for all $k \in \Z$ and $i \in N_{k+1}$. Since $\# N_{k+1}^j \ge \# N_{k+1}^j(\lambda) = K$ we have $\eps_{k,j} \le ((K+1)K)^{-1}$ and thus
  \begin{equation} \label{eq:mu0-arvio}
    \mu_0(Q_{k+1,i}) \le \roo^t \mu_0(Q_{k,j})
  \end{equation}
  for all $k \in \Z$ and $i \in N_{k+1}$.

  The measure $\mu$ is now defined using the Carath\'eodory construction for the set function $\mu_0$: For each $\delta>0$ and $A \subset X$ first define
  \begin{equation*}
    \mu_\delta(A) = \inf\biggl\{ \sum_{i=1}^k \mu_0(Q_i) : A \subset \bigcup_{i=1}^k Q_i, \text{ where } Q_i \in \QQ \text{ such that } \diam(Q_i) \le \delta \biggr\}
  \end{equation*}
  and then let $\mu(A) = \sup_{\delta>0} \mu_\delta(A)$. Since
  \begin{equation*}
    \mu_0(Q_{k,j}) = \sum_{i \in N_{k+1}^j} \mu_0(Q_{k+1,i})
  \end{equation*}
  for all $k \in \Z$ and $j \in N_k$, the Borel regular outer measure $\mu$ satisfies $\mu(A) = \mu_\delta(A)$ for all $\delta>0$ and $\mu(Q_{k,j}) = \mu_0(Q_{k,j})$ for all $k \in \Z$ and $j \in N_k$.
  
  Let $x \in X$ and $0<r<R$ and take $n,N \in \Z$ such that $\roo^{n-1} \le R < \roo^{n-2}$ and $\roo^N \le r < \roo^{N-1}$. Since the measure $\mu$ is clearly supported on $X$, it suffices to show that
  \begin{equation*}
    \mu(B(x,r)) \le \tilde{C} \roo^{(N-n+2)t} \mu(B(x,R))
  \end{equation*}
  for some constant $\tilde{C}>0$, independent of $x$ and $R$. We may assume that $n<N-2$, since otherwise the above estimate holds with $\tilde{C}=\roo^{-4t}$.
  Let us first consider the case where
  \begin{equation} \label{eq:pienta-jakoon}
    Q_{k+1}(x,r) \subset \bigcup_{j \in N_k} N_{k+1}^j \setminus N_{k+1}^j(\lambda)
  \end{equation}
  for all $k \in \{ n,\ldots,N-2 \}$. For a fixed number $k$ this means that each $j \in Q_k(x,r)$ is such that $\mu(Q_{k+1,i}) = \eps_{k,j} \mu(Q_{k,j})$ for all $i \in N_{k+1}^j \cap Q_{k+1}(x,r)$. Since a cube $Q_{k,j}$ has at most $\# N_{k+1}^j$ subcubes, we have
  \begin{equation} \label{eq:pienta-arvio}
    \sum_{i \in Q_{k+1}(x,r)} \mu(Q_{k+1,i}) \le \sum_{j \in Q_k(x,r)} \# N_{k+1}^j \eps_{k,j} \mu(Q_{k,j}) \le \roo^t \sum_{j \in Q_k(x,r)} \mu(Q_{k,j})
  \end{equation}
  by the choice of $K$. Therefore, multiple applications of \eqref{eq:pienta-arvio} and Proposition~\ref{prop:dyadic}(3) imply
  \begin{equation} \label{eq:pienta-koko-ajan-arvio}
  \begin{split}
    \mu(B(x,r)) &\le \sum_{i \in Q_{N-2}(x,r)} \mu(Q_{N-2},i) \le \roo^{t(N-2-n)} \sum_{i \in Q_n(x,r)} \mu(Q_{n,i}) \\ &\le \roo^{t(N-2-n)} \mu(B(x,r+ 2C \roo^n)) \le \roo^{t(N-2-n)} \mu(B(x,R))
  \end{split}
  \end{equation}
  since $r+ 2C \roo^n < \roo^{n-1}(\roo^2+2C\roo) < \roo^{n-1} \le R$ by~\eqref{eq:est3} and the fact that $r<\roo^{N-1}\le \roo^{n+1}$.

  Let us then assume that \eqref{eq:pienta-jakoon} does not hold for some number in $\{ n,\ldots,N-2 \}$, and let $k$ be the largest such number in $\{ n,\ldots,N-2 \}$. This means that~\eqref{eq:pienta-jakoon} is satisfied for all numbers in $\{ k+1,\ldots,N-2 \}$ and
  \begin{equation*}
    \mu(Q_{k+1},i) = \Bigl( \frac{1}{K+1} + \eps_{k,j} \Bigr) \mu(Q_{k,j})
  \end{equation*}
  for some $j \in Q_k(x,r)$ and $i \in N_{k+1}^j(\lambda) \cap Q_{k+1}(x,r)$. We will show that now
  \begin{equation} \label{eq:isoa-jakoon}
    Q_{N,h} \subset Q_{k,j}
  \end{equation}
  for all $h \in Q_N(x,r)$, and for this it suffices that $x_{N,h}\in Q_{k,j}$. Fix $h \in Q_N(x,r)$ and 
  let $i_0 \in Q_N(x,r)$ be such that $Q_{N,i_0} \subset Q_{k+1,i}$. Then
  \begin{equation} \label{eq:kolmio1}
    |x_{N,h}-x_{N,i_0}| \le |x_{N,h}-x| + |x-x_{N,i_0}| \le 2(\roo^{N-1} + C\roo^N) \le 4\roo^{N-1}
  \end{equation}
  by Proposition~\ref{prop:dyadic}(3), the choice of $N$, and the fact that $C\roo < 1$. Since $i \in N_{k+1}^j(\lambda)$ and $Q_{N,i_0} \subset Q_{k+1,i}$, we have
  \begin{equation} \label{eq:kolmio2}
    |x_{N,i_0}-x_{k,j}| \le |x_{N,i_0}-x_{k+1,i}| + |x_{k+1,i}-x_{k,j}| \le C\roo^{k+1} + \lambda\roo^k + C\roo^{k+1}.
  \end{equation}
  Thus, putting \eqref{eq:kolmio1} and \eqref{eq:kolmio2} together and recalling that $k \le N-2$, estimate \eqref{eq:est2} gives
  \begin{equation*}
    |x_{N,h}-x_{k,j}| \le (2C+4)\roo^{k+1} + \lambda\roo^k < c\roo^k
  \end{equation*}
  which shows that \eqref{eq:isoa-jakoon} holds. Now, if $h \in N_n$ is such that $Q_{k,j} \subset Q_{n,h}$, then \eqref{eq:pienta-arvio}, \eqref{eq:isoa-jakoon}, \eqref{eq:mu0-arvio}, Proposition~\ref{prop:dyadic}(3), and \eqref{eq:est3}
  (as in~\eqref{eq:pienta-koko-ajan-arvio}) imply
  \begin{align*}
    \mu(B(x,r)) &\le \sum_{i \in Q_{N-1}(x,r)} \mu(Q_{N-1},i) \le \roo^{t(N-2-k)} \sum_{i \in Q_{k+1}(x,r)} \mu(Q_{k+1},i) \\
    &\le \roo^{t(N-2-k)} \mu(Q_{k,j}) \le \roo^{t(N-2-k)} \roo^{t(k-n)} \mu(Q_{n,h}) \\ 
    &\le \roo^{t(N-2-n)} \mu(B(x,r+ 2C \roo^n)) \le \roo^{t(N-2-n)} \mu(B(x,R)),
  \end{align*}
  showing that indeed $\ldimreg(\mu) \ge t$.
\end{proof}

If $X$ is a complete metric space, then it follows from~\cite[Theorem~1]{LuukkainenSaksman1998} that $X$ is doubling if and only if there exists a doubling measure $\mu$ supported on $X$. The corresponding consequence of Theorem~\ref{thm:finite-condition} is the following corollary. Its proof follows immediately from Theorem~\ref{thm:finite-condition} by recalling that, by \cite[Lemma~2.1]{KaenmakiLehrbackVuorinen2013}, $X$ is uniformly perfect if and only if $\diml(X)>0$, and that $\ldimreg(\mu)>0$ if and only if $\mu$ is reverse-doubling and supported on $X$.

\begin{corollary}
  Suppose that $X$ is a complete metric space satisfying the finite doubling property. Then $X$ is uniformly perfect if and only if there exists a reverse-doubling measure $\mu$ supported on $X$.
\end{corollary}

We shall next show that a slight modification of the proof of Theorem~\ref{thm:finite-condition} gives an alternative proof for the ``lower'' part of \cite[Theorem~9]{BylundGudayol2000}.

\begin{theorem} \label{thm:doubling}
  Suppose that $X$ is a uniformly perfect complete doubling metric space. Then for each $\eps>0$ there exists a doubling measure $\mu$ supported on $X$ such that $\ldimreg(\mu) > \diml(\mu)-\eps$.
\end{theorem}

\begin{proof}
We will show that the measure $\mu$ in Theorem \ref{thm:finite-condition} can be defined to be doubling if the space $X$ is doubling. Unfortunately the proof of Theorem \ref{thm:finite-condition} does not seem to give this directly. From the technical point of view, to show that the measure $\mu$ is doubling, we cannot allow $\eps_{k,j}$ in the annulus region of \eqref{eq:mu0-def} to depend on $k$ and $j$. We use the notation introduced in the proof of Theorem \ref{thm:finite-condition}.

If $X$ is doubling, then there exists $M\in\N$ such that $N^j_{k+1}\le M$ for all $k\in\Z$ and $j\in N_k$. We make the following small change to the definition of the measure $\mu$: Take $\eps=\frac 1 {(K+1)M}$, and define the pre-measure $\mu_0$ by first setting $\mu_0(Q_{0,i_0})=1$ for a fixed $i_0 \in N_0$ and then requiring that
 \begin{equation*}
   \mu_0(Q_{k+1,i}) =
   \begin{cases}
     \bigl(\frac{1}{K+1}+\tilde\eps_{k,j}\bigr) \mu_0(Q_{k,j}), &\text{if } i \in N_{k+1}^j(\lambda) \\
     \eps \mu_0(Q_{k,j}), &\text{if } i \in N_{k+1}^j \setminus N_{k+1}^j(\lambda)
   \end{cases}
 \end{equation*}
 for all $k \in \Z$ and $i \in N_{k+1}$, where the numbers $\tilde\eps_{k,j}$ are now chosen in such a way that 
 \begin{equation*}
   \mu_0(Q_{k,j}) = \sum_{i \in N_{k+1}^j} \mu_0(Q_{k+1,i})
 \end{equation*}
 for all $k \in \Z$ and $j \in N_k$. Then in particular $0<\tilde\eps_{k,j}\le \frac 1 {K(K+1)}$ 
 and
 \begin{equation} \label{eq:mu0-arvio-tupl}
   \eps \mu_0(Q_{k,j}) \le \mu_0(Q_{k+1,i}) \le \roo^t \mu_0(Q_{k,j})
 \end{equation}
 for all $k \in \Z$ and $i \in N_{k+1}$. The measure $\mu$ is then defined just as above, and the rest of the proof works as such, showing that $\mu$ satisfies~\eqref{eq:lower_reg} with $t$, and so we only need to show that $\mu$ is doubling.

 Fix $x\in X$ and $r>0$. Let $N\in\Z$ be such that $\roo^{N}\le r<\roo^{N-1}$ and let $k_0\le N-2$ be the largest integer for which there exists $j\in N_{k_0}$ such that $Q_{N+1,h}\subset Q_{k_0,j}$ for all $h\in Q_{N+1}(x,2r)$. Notice that such $k_0\in\Z$ exists by Proposition~\ref{prop:dyadic}(4). 
 Then~\eqref{eq:pienta-jakoon} holds for all $k\in\{k_0+1,\dots,N-2\}$, since otherwise we would obtain, following the reasoning in the proof of Theorem~\ref{thm:finite-condition} from the case where~\eqref{eq:pienta-jakoon} did not hold for all $k\in\{n,\dots,N-2\}$, some $k\in\{k_0+1,\dots,N-2\}$ such that $Q_{N+1,h}\subset Q_{k,j}$ for all $h\in Q_{N+1}(x,2r)$, contradicting the choice of $k_0$.

 Next we show that the measures of the cubes in $Q_{N+1}(x,2r)$ are always comparable (with uniform constants). To this end, let $i,j\in Q_{N+1}(x,2r)$, and let $i_k, j_k\in Q_{k}(x,r)$, for $k\in \{ k_0,\ldots,N \}$, be such that $Q_{N+1,i}\subset Q_{k,i_k}$ and $Q_{N+1,j}\subset Q_{k,j_k}$. Since~\eqref{eq:pienta-jakoon} holds for all $k\in\{ k_0+1,\ldots,N-2 \}$, for these $k$ we have that $\mu(Q_{k+1,i_{k+1}}) = \eps\mu(Q_{k,i_{k}})$ (and the same for $j$). Using this together with estimate~\eqref{eq:mu0-arvio-tupl} and the fact that $Q_{k_0,i_{k_0}}=Q_{k_0,j_{k_0}}$, we obtain 
 \begin{equation}\label{eq:vertaa}\begin{split}
  \mu(Q_{N+1,i}) & \le \roo^{3t}\mu(Q_{N-2,i_{N-2}}) = \roo^{3t}\eps^{N-k_0-3}\mu(Q_{k_0+1,i_{k_0+1}})
  \le \roo^{4t}\eps^{N-k_0-3}\mu(Q_{k_0,j_{k_0}}) \\ 
  & \le \roo^{4t}\eps^{N-k_0-4}\mu(Q_{k_0+1,j_{k_0+1}})
  =  \roo^{4t}\eps^{-1}\mu(Q_{N-2,j_{N-2}}) \le \roo^{4t}\eps^{-4}\mu(Q_{N+1,j}).
 \end{split}
 \end{equation}

 Finally, choose $i_0\in Q_{N+1}(x,r)$ so that $x\in Q_{N+1,i_0}$. Since $2C\roo<1$ by~\eqref{eq:est3}, we have for all $z\in Q_{N+1,i_0}$ that $|x-z|\le |x-x_{N+1,i_0}|+|x_{N+1,i_0}-z|\le 2C\roo^{N+1}<\roo^{N}$, and thus $Q_{N+1,i_0}\subset B(x,r)$. Since $X$ is doubling, there exists $M_0>0$, independent of $x$ and $r$, such that $\# Q_{N+1}(x,2r)\le M_0$. Using~\eqref{eq:vertaa} with $j=i_0$, we thus obtain that
 \begin{equation*}
  \mu(B(x,r))\ge \mu(Q_{N+1,i_0}) \ge M_0^{-1} \roo^{-4t}\eps^{4} \sum_{i\in Q_{N+1}(x,2r)} \mu(Q_{N+1,i})
  \ge \tilde c\mu(B(x,2r)),
 \end{equation*}
 where $\tilde c = M_0^{-1} \roo^{-4t}\eps^{4}>0$ is independent of $x$ and $r$. This shows that $\mu$ is doubling, and the proof is complete.
\end{proof}

\section{Inhomogeneous self-similar sets}\label{sect:inhomog}

A finite collection $\{ \fii_i \}_{i=1}^\kappa$ of contractive similitudes acting on $\R^d$ is called an \emph{(similitude) iterated function system (IFS)}. Hutchinson \cite{Hutchinson1981} proved that for a given IFS there exists a unique non-empty compact set $E \subset \R^d$ such that
\begin{equation*}
  E = \bigcup_{i=1}^\kappa \fii_i(E).
\end{equation*}
The set $E$ is the \emph{self-similar set} associated to the IFS.
Furthermore, if $C \subset \R^d$ is compact, then there exists a unique non-empty compact set $E_C \subset \R^d$ such that
\begin{equation*}
  E_C = \bigcup_{i=1}^\kappa \fii_i(E_C) \cup C.
\end{equation*}
The set $E_C$ is called the \emph{inhomogeneous self-similar set with condensation $C$}. Such sets were introduced in \cite{BarnsleyDemko1985} and studied in \cite{Barnsley2006}. Observe that $E_\emptyset$ is the self-similar set $E$. It was proved in \cite[Lemma 3.9]{Snigireva2008} that
\begin{equation} \label{eq:inhomog-eq}
  E_C = E \cup \bigcup_{\iii \in \Sigma_*} \fii_\iii(C).
\end{equation}
Here the set $\Sigma_*$ is the set of all finite words $\{ \varnothing \} \cup \bigcup_{n \in \N} \Sigma_n$, where $\Sigma_n = \{ 1,\ldots,\kappa \}^n$ for all $n \in \N$ and $\varnothing$ satisfies $\varnothing\iii = \iii\varnothing = \iii$ for all $\iii \in \Sigma_*$. For notational convenience, we set $\Sigma_0 = \{ \varnothing \}$. The set $\Sigma = \{ 1,\ldots,\kappa \}^\N$ is the set of all infinite words. If $\iii = i_1 \cdots i_n \in \Sigma_*$ for some $n \in \N$, then $\fii_\iii = \fii_{i_1} \cdots \fii_{i_n}$. We define $\fii_\varnothing$ to be the identity mapping. The concatenation of two words $\iii \in \Sigma_*$ and $\jjj \in \Sigma_* \cup \Sigma$ is denoted by $\iii\jjj \in \Sigma_* \cup \Sigma$ and the length of $\iii \in \Sigma_* \cup \Sigma$ is denoted by $|\iii|$. If $\jjj \in \Sigma_* \cup \Sigma$ and $1 \le n < |\jjj|$, then we define $\jjj|_n$ to be the unique word $\iii \in \Sigma_n$ for which $\iii\kkk = \jjj$ for some $\kkk \in \Sigma_*$. We also set $\iii^- = \iii|_{|\iii|-1}$ for all $\iii \in \Sigma_* \setminus \{ \varnothing \}$.

From \eqref{eq:inhomog-eq} we see immediately that
\begin{align*}
  \dimh(E_C) &= \max\{ \dimh(E), \dimh(C) \}, \\
  \dimp(E_C) &= \max\{ \dimp(E), \dimp(C) \},
\end{align*}
where $\dimh$ denotes the Hausdorff dimension and $\dimp$ is the packing dimension. Fraser \cite[Corollaries 2.2 and 2.5]{Fraser2012} showed that if the IFS satisfies the open set condition, then
\begin{align*}
  \udimm(E_C) &= \max\{ \udimm(E), \udimm(C) \}, \\
  \max\{ \ldimm(E), \ldimm(C) \} \le \ldimm(E_C) &\le \max\{ \ldimm(E), \udimm(C) \},
\end{align*}
where $\udimm$ and $\ldimm$ are the upper and lower Minkowski dimensions, respectively. Furthermore, Baker, Fraser, and M\'ath\'e \cite{BakerFraserMathe2017} demonstrated that the equality above may fail without the open set condition.

In the following two theorems, we obtain analogous results for the Assouad and lower dimensions of inhomogeneous self-similar sets. However, unlike in the case of Minkowski dimensions, we need to pose an additional separation condition for the condensation set; Example~\ref{ex:no_separation} below illustrates the necessity of this extra assumption.

We say that the IFS satisfies the \emph{condensation open set condition (COSC) with condensation} $C$, if there exists an open set $U$ such that $C \subset U \setminus \bigcup_i \overline{\fii_i(U)}$,
\begin{equation*}
  \fii_i(U) \subset U \text{ for all $i$} \quad \text{and} \quad \fii_i(U) \cap \fii_j(U) \ne \emptyset \text{ whenever $i \ne j$}.
\end{equation*}
Without the reference to the condensation set $C$, the above is the familiar open set condition. The COSC is a slight modification of the inhomogeneous open set condition introduced in \cite[\S 4.3.1]{Snigireva2008}. It is well known that if the open set condition is satisfied, then the self-similar set $E$ is $s$-regular for $s$ satisfying $\sum_i \Lip(\fii_i)^s = 1$ and hence $\diml(E) = \dima(E) = s$. Here $\Lip(\fii_i)$ is the contraction coefficient of the similitude $\fii_i$.

\begin{theorem}\label{thm:dim of inhomog}
  Let $C \subset \R^d$ be a compact set and let $\{ \fii_i \}$ be a similitude IFS satisfying the COSC with condensation $C$. If $E$ is the associated self-similar set and $E_C$ is the inhomogeneous self-similar set with condensation $C$, then
  \begin{equation*}
    \dima(E_C) = \max\{ \dima(E), \dima(C) \}.
  \end{equation*}
\end{theorem}

\begin{proof}
  Since the Assouad dimension is monotone, it suffices to show that $$\dima(E_C) \le \max\{ \dima(E), \dima(C) \}.$$ Assuming that $\dima(C) < \dima(E)$, we will show that $\dima(E_C) \le \dima(E)$. Recall that, by the open set condition, $\gamma = \dima(E) = \diml(E)$ satisfies $\sum_i \Lip(\fii_i)^\gamma = 1$. Fix $x \in E_C$ and $0<r<R < \diam(E_C)$, and let $\BB$ be a maximal $r$-packing of $E_C \cap B(x,R)$. It suffices to show that for each $\gamma' > \gamma$ there is a constant $C$ not depending on $x$ such that
  \begin{equation}\label{eq:ylaraja}
    \#\BB \le C\Bigl( \frac{r}{R} \Bigr)^{-\gamma'}
  \end{equation}
  for all $0<r<R$, since then $\dima(E_C) \le \gamma = \dima(E)$.
  
  Let $U$ be the open set from the COSC. Since $C \subset U \setminus \bigcup_i \overline{\fii_i(U)}$ is compact, we have $d = \dist(C,\bigcup_i \overline{\fii_i(U)} \cup (\R^d \setminus U)) > 0$. 
  It then follows from the COSC that 
  \begin{equation}\label{eq:kaukana}
    \dist(\fii_\iii(C),E_C\setminus \fii_\iii(C))\ge d\Lip(\fii_\iii)                                                                                                                                                                                                                                                                            
  \end{equation}
  for all $\iii\in\Sigma_*$. Let $\dima(C) < \lambda < \dima(E)$, and assume first that $\fii_\iii(C)\cap B(x,R)\neq\emptyset$ for some $\iii\in\Sigma_*$ with $R<d\Lip(\fii_\iii)$. Then it is obvious from~\eqref{eq:kaukana} that $E_C \cap B(x,R)=\fii_\iii(C)\cap B(x,R)$. Since $\fii_\iii(C)$ is a dilated copy of $C$ and $\dima(C)<\lambda<\gamma$, it follows that 
  \begin{equation*}
    \#\BB \le C\Bigl( \frac{r}{R} \Bigr)^{-\lambda} \le C\Bigl( \frac{r}{R} \Bigr)^{-\gamma},
  \end{equation*}
  where $C>0$ is independent of $x$, $r$, and $R$. We may hence assume that $\fii_\iii(C)\cap B(x,R) = \emptyset$ for all $\iii\in\Sigma_*$ for which $R<d\Lip(\fii_\iii)$.
  
  Define
  \begin{align*}
    \BB_1 &= \{ B \in \BB : E \cap B = \emptyset \}, \\
    \BB_2 &= \BB \setminus \BB_1 = \{ B \in \BB : E \cap B \ne \emptyset \}.
  \end{align*}
  Observe that if $B \in \BB_1$, then the center of $B$ is in $\bigcup_{\iii \in \Sigma_*} \fii_\iii(C)$. 
  Since the set $\bigcap_{n=1}^\infty \fii_{\iii|_n}(E_C)$ contains exactly one point of $E$ for all $\iii\in\Sigma$ and
  \begin{equation}\label{eq:raja-fii-C}
    \diam(\fii_\iii(C)) \le \diam(\fii_\iii(E_C)) = \Lip(\fii_\iii) \diam(E_C)
  \end{equation}
  for all $\iii \in \Sigma_*$, we see that if $\iii \in \Sigma_*$ is such that $\Lip(\fii_\iii) < r/\diam(E_C)$, then $\fii_\iii(C)$ is contained in the $r$-neighborhood of $E$. Note also that if $B \in \BB$ is such that the center of $B$ is in the $r$-neighborhood of $E$, then $B \in \BB_2$. Writing
  \begin{equation*}
    \BB_1(\iii) = \{ B \in \BB : \text{the center of $B$ is in } \fii_\iii(C) \}
  \end{equation*}
  for all $\iii \in \Sigma_*$, we thus have
  \begin{equation}\label{eq:BB-inclusion}
    \BB_1 \subset \bigcup \{ \BB_1(\iii) : \Lip(\fii_\iii) \ge r/\diam(E_C) \}.
  \end{equation}
  Since $\BB_1(\iii)$ is an $r$-packing of $\fii_\iii(C)$, which is a dilated copy of $C$, it follows from the definition of $\dima$ and estimate~\eqref{eq:raja-fii-C} that
  \begin{equation*}
    \#\BB_1(\iii) \le C\biggl( \frac{r}{\diam(\fii_\iii(C))} \biggr)^{-\lambda} \le C\biggl( \frac{r}{\Lip(\fii_\iii) \diam(E_C)} \biggr)^{-\lambda}
  \end{equation*}
  for all $\iii \in \Sigma_*$, where the constant $C$ is independent of $\iii$.
  
  Let $\ualpha = \max_i \Lip(\fii_i)$ and $\lalpha = \min_i \Lip(\fii_i)$. Choose $m,M\in\Z$ such that 
  $d\lalpha\ualpha^m \le R < d\lalpha\ualpha^{m-1}$
  and $\ualpha^{M}\diam(E_C) \le r < \ualpha^{M-1}\diam(E_C)$,
  and define
  \begin{equation*}
    N(\roo) = \{ \iii \in \Sigma_* : \Lip(\fii_\iii) \le \roo < \Lip(\fii_{\iii^-}) \}
  \end{equation*}
  for all $\roo>0$. Observe that if $\iii \in N(\ualpha^{m-k})$ for some $k \in \N$, then
  \begin{equation*}
    R < d\lalpha \ualpha^{m-1} \le d\lalpha \ualpha^{m-k} < d\lalpha \Lip(\fii_{\iii^-}) \le d \Lip(\fii_\iii),
  \end{equation*}
  and hence $B(x,R)$ does not intersect $\fii_{\iii}(C)$. 
  On the other hand, if $\iii \in N(\ualpha^{M+k})$ for some $k \in \N$, then
  \begin{equation*}
    \Lip(\fii_\iii) \le \ualpha^{M+k} < \ualpha^{M} \le r/\diam(E_C)
  \end{equation*}
  which means that such a word $\iii$ is not used in \eqref{eq:BB-inclusion}. We can thus conclude that the inclusion~\eqref{eq:BB-inclusion} can written as follows:
  \begin{equation*}
    \BB_1 \subset \bigcup_{\ell=m}^M \bigcup_{\iii \in N(\ualpha^\ell)} \BB_1(\iii).
  \end{equation*}
  The role of $\ualpha$ is to just give the slowest contraction speed, so that we get everything in the above union.

  When $\iii \in N(\ualpha^m)$, we have that $\diam(\fii_\iii(U))=\Lip(\fii_\iii)\diam(U)\approx R$, and hence $B(x,R)$ can intersect at most $K_1$ many open sets $\fii_\iii(U)$ with $\iii \in N(\ualpha^m)$. Since $C \subset U \setminus \bigcup_i \overline{\fii_i(U)}$, we thus obtain 
  \begin{equation*}
    \#\BB_1 \le CK_1 \sum_{\ell=m}^M \# N(\ualpha^{\ell-m})\biggl( \frac{r}{\diam(E_C) \ualpha^\ell} \biggr)^{-\lambda} \le CK_1 \sum_{\ell=m}^M \# N(\ualpha^{\ell-m}) \ualpha^{(\ell-M)\lambda}.
  \end{equation*}
  Since $\gamma = \dima(E) = \diml(E)$ and the open set condition is satisfied, we have
  \begin{equation*}
    1 = \sum_{\iii \in N(\roo)} \Lip(\fii_\iii)^\gamma 
    \begin{cases}
      \le \# N(\roo) \roo^\gamma, \\
      > \lalpha^\gamma \# N(\roo) \roo^\gamma,
    \end{cases}
  \end{equation*}
  and hence $\# N(\roo) \approx \roo^{-\gamma}$ for all $\roo>0$. With this we continue:
  \begin{align*}
    \sum_{\ell=0}^{M-m} \# N(\ualpha^\ell) \ualpha^{(\ell+m-M)\lambda} &\lesssim \sum_{\ell=0}^{M-m} \ualpha^{-\ell\gamma} \ualpha^{(\ell+m-M)\lambda} = \ualpha^{-(M-m)\lambda} \sum_{\ell=0}^{M-m} \ualpha^{-(\gamma-\lambda)\ell} \\
    &\le C \ualpha^{-(M-m)\lambda} \ualpha^{-(\gamma-\lambda)(M-m)}\\ 
    & = C \ualpha^{-(M-m)\gamma} \le C \biggl( \frac{r\ualpha \diam(E_C)^{-1}}{R(\lalpha d)^{-1}} \biggr)^{-\gamma}.
  \end{align*}
  Therefore we have $\#\BB_1 \le C\bigl( \frac{r}{R} \bigr)^{-\gamma}$ for some constant $C>0$.
  
  Observe that there exists $K_2\in\N$ such that each $\fii_\iii(E)$, where $\iii \in N(r)$, intersects at most $K_2$ many sets $B \in \BB$. Therefore,
  \begin{equation*}
    \#\BB_2 \le K_2 \cdot \#\{ \iii \in N(r) : \fii_\iii(E) \cap B \neq \emptyset \text{ for some } B \in \BB \}.
  \end{equation*}
  If $\iii \in N(r)$, then $\fii_\iii(E)$ is contained in a ball of radius $r \diam(E)$ centered at any point of $\fii_\iii(E)$. Recall that, by \cite[Corollary 4.8]{KaenmakiVilppolainen2008}, there exist $z \in E$ and $\delta>0$ such that the collection $\{ B(\fii_\iii(z),\delta r) : \iii \in N(r) \}$ of balls is disjoint. Therefore, by \cite[Lemma 2.1(4)]{KaenmakiRajalaSuomala2012}, the collection $\{ B(\fii_\iii(z),r \diam(E)) : \iii \in N(r) \text{ and } \fii_\iii(z)\in E \cap B(x,2R) \}$ can be divided into $K_3$ many disjoint subcollections. Since each such disjoint subcollection is an $(r \diam(E))$-packing of $E \cap B(x,2R)$, its cardinality, for each $\gamma' > \gamma$, is bounded from above by a constant times $\bigl( \frac{r}{R} \bigr)^{-\gamma'}$; recall \cite[Lemma 2.1(3)]{KaenmakiRajalaSuomala2012}. Hence for every $\gamma'>\gamma$ there is a constant $C$ such that
  $\#\BB_2 \le C\bigl( \frac{r}{R} \bigr)^{-\gamma'}$. Since $\#\BB=\#\BB_1 + \#\BB_2$, the claim~\eqref{eq:ylaraja} follows.
  
  Let us then assume that $\gamma = \dima(E) \le \dima(C)$, and let $\lambda > \dima(C)$. We need to show that $\dima(E_C) \le \lambda$. We can follow the reasoning from the above case, but now $\gamma < \lambda$, and so we obtain that
  \begin{align*} 
    \#\BB_1 &\le C\ualpha^{-(M-m) \lambda} \sum_{\ell=0}^{M-m} \ualpha^{-(\gamma-\lambda)\ell} 
    \le C \ualpha^{-(M-m) \lambda} 
    \le C \biggl( \frac{r\ualpha \diam(E_C)^{-1}}{R(\lalpha d)^{-1}} \biggr)^{-\lambda}
  \end{align*}
  for some constant $C>0$.
  Since the reasoning concerning $\BB_2$ works as above, for $\gamma'=\lambda>\gamma$, we get
  \begin{equation*}
    \#\BB = \#\BB_1 + \#\BB_2 \le C\Bigl( \frac{r}{R} \Bigr)^{-\lambda} + C\Bigl( \frac{r}{R} \Bigr)^{-\lambda} \le 2C\Bigl( \frac{r}{R} \Bigr)^{-\lambda}.
  \end{equation*}
  Therefore $\dima(E_C) \le \dima(C) = \max\{ \dima(E), \dima(C) \}$ as claimed.
\end{proof}

\begin{theorem}\label{thm:dim of inhomog2}
  Let $C \subset \R^d$ be a nonempty compact set and let $\{ \fii_i \}$ be a similitude IFS satisfying the COSC with condensation $C$. If $E$ is the associated self-similar set and $E_C$ is the inhomogeneous self-similar set with condensation $C$, then
  \begin{equation*}
    \diml(E_C) = \diml(C).
  \end{equation*}
\end{theorem}

\begin{proof}
  By the COSC, we have $C \cap \bigcup_i \fii_i(E_C) = \emptyset$. Therefore, by \cite[Theorem 2.2]{Fraser2014}, $\diml(E_C) = \min\{ \diml(E_C), \diml(C) \} \le \diml(C)$. 
  To show the opposite inequality, we may clearly assume that $\diml(C)>0$. Let $0<t<\diml(C)$ and fix $x\in E_C$ and $0<r<R \le \diam(E_C)/2$, and let $\BB$ be a cover of $E_C\cap B(x,R)$ consisting of balls of radius $r$; for $\diam(E_C)/2\le R <\diam(E_C)$ the claim then easily follows from the case $R=\diam(E_C)/2$. Assume first that $x\in E$. Choose $\iii \in \Sigma_*$ such that  $x\in \fii_\iii(E_C)\subset B(x,R)$ but $\fii_{\iii^-}(E_C) \setminus B(x,R) \ne \emptyset$.
  Then $\diam(\fii_{\iii}(C)) = \diam(C)\diam(\fii_{\iii}(E_C))/\diam(E_C) \ge cR$, and since $\fii_{\iii}(C)$ is a dilated copy of $C$ and $\BB$ is in particular a cover of $\fii_{\iii}(C)$, we obtain that
  \begin{equation*}
    \#\BB\ge c\biggl(\frac r {\diam(\fii_\iii(C))}\biggr)^{-t} \ge c\Bigl(\frac r {R}\Bigr)^{-t}, 
  \end{equation*}
  where the constant $c>0$ is independent of $x$, $R$, and $r$. 

  Assume then that $x\in \fii_\iii(C)$ for some $\iii \in \Sigma_*$. The claim is clear if 
  $\diam(\fii_\iii(C))\ge c_1 R$, where $c_1=\diam(C)/(2\diam(E_C))$, since $\BB$ is a cover of $\fii_{\iii}(C)\cap B(x,R)$ and hence $\#\BB\ge c\bigl(\frac r {R}\bigr)^{-t}$. 
  We may thus assume that $\diam(\fii_\iii(C))\le c_1 R$, whence
  \[\diam(\fii_\iii(E_C))=\frac{\diam(E_C)}{\diam(C)}\diam(\fii_\iii(C))\le R/2.\]
  Now choose $x'\in E\cap\fii_\iii(E_C)$. Then $x'\in B(x,R/2)$ and $\BB$ is a cover of $E_C\cap B(x',R/2)\subset E_C\cap B(x,R)$,
  and so the previous case where the ball is centered in $E$, applied for $x'\in E$, yields that
  \begin{equation*}
    \#\BB\ge c\Bigl(\frac r {R/2}\Bigr)^{-t} \approx \Bigl(\frac r {R}\Bigr)^{-t}.
  \end{equation*}
  Thus $\diml(E_C)\ge t$, and so the claim $\diml(E_C)\ge\diml(C)$ follows.
\end{proof}

We finish this section by examining the COSC assumption in the previous theorems.

\begin{example}\label{ex:no_separation}
  Let us first recall that, by Theorem \ref{thm:dim of inhomog}, the dimension formula $\dima(E_C) = \max\{ \dima(E), \dima(C) \}$ holds provided that the COSC holds, i.e.\ there is lots of separation. Interestingly, the formula holds in the real line also when there is lots of overlap, since Fraser, Henderson, Olson, and Robinson \cite[Theorem 3.1]{FraserHendersonOlsonRobinson2015} showed that if the weak separation condition does not hold, and not all the mappings have the same fixed point, then $\dima(E)=1$. In this case, the formula follows from the monotonicity of the Assouad dimension. Note that the inequality $\dima(E_C)\ge\max\{ \dima(E), \dima(C) \}$ is always valid.

  The following construction shows that the open set condition is not enough for the formula to hold. Consider $\fii_1,\fii_2\colon\R\to\R$, $\fii_1(x)=\frac 1 3 x$ and $\fii_2(x)=\frac 1 3 x + \frac 2 3$,
 so that the self-similar set $E$ is the usual $\frac 1 3$-Cantor set and $\dima(E)=\frac{\log 2}{\log 3}$. Define first
 \[
 C=\bigl\{\bigl(1+\tfrac 1 {j+1}\bigr)3^{-j} : j\in\N \bigr\}\cup\{0\},
 \]
 and consider the inhomogeneous self-similar set $E_C$ with condensation $C$; 
 it is immediate that $E$ satisfies the open set condition but $E_C$ does not satisfy the COSC. It is easy to see that $\dima(C)=0$. On the other hand, $E_C$ contains for each $k\in\N$ a dilated and translated copy of the set $\{\frac 1 j : j \in \{ 2,3,\ldots,k \}\}$, and so the set $W=\{\frac 1 j : j \in \{ 2,3,\ldots \}\}$ is a weak tangent of $E_C$; see~\cite[Section~6]{MackayTyson2010} for the definition. Thus it follows from~\cite[Proposition~6.1.5]{MackayTyson2010} that $\dima(E_C)\ge\dima(W)=1$. In particular $\dima(E_C)=1>\max\{ \dima(E), \dima(C) \}$, and so the conclusion of Theorem~\ref{thm:dim of inhomog} does not hold.

 Let us then focus on the lower dimension. Let $\fii_1,\fii_2$ be as above, but consider now the condensation set $C=[4/10,6/10]\cup\{1/6\}$. Then $\diml(E)=\frac{\log 2}{\log 3}$ and $\diml(C)=0$, and $E_C$ does not satisfy the COSC due to overlaps. Moreover, it is easy to see that each ball $B(x,r)$ with $x\in E_C$ and $0<r\le 1$ contains an interval of length $cr$, where $0<c<1$ is independent of $x$ and $r$. Thus $\diml(E_C)=1>0=\diml(C)$, showing that the conclusion of Theorem~\ref{thm:dim of inhomog2} does not hold. Observe, however, that the inequality $\diml(E_C)\ge\diml(C)$ is always valid. This is not a triviality since $\diml$ is not monotone, but it can be seen from the proof of Theorem~\ref{thm:dim of inhomog2}.
\end{example}

\section{Regularity of measures on inhomogeneous self-similar sets}

Vol'berg and Konyagin~\cite[Theorem~4]{VolbergKonyagin1987} gave an example of a compact set $X\subset\R^n$ such that $\dima(X) < \udimreg(\mu)$ for all doubling measures $\mu$ supported on $X$. Their set $X$ was an inhomogeneous self-similar set with a single point condensation satisfying the strong separation condition, and all the similitudes in the construction had the same contraction coefficient. The result of~\cite[Theorem~4]{VolbergKonyagin1987} is a special case of the following theorem, which reveals a general phenomenon behind the construction of Vol'berg and Konyagin.

\begin{theorem}\label{thm:upper_reg}
  Let $C \subset \R^d$ be a nonempty compact set, let $\{ \fii_i \}$ be a similitude IFS satisfying the COSC with condensation $C$, and let $E_C$ be the inhomogeneous self-similar set with condensation $C$. If the associated self-similar set $E$ is such that $\dima(C) < \dima(E)$, then $\dima(E_C) < \udimreg(\mu)$ for all measures $\mu$ supported on $E_C$.
\end{theorem}

\begin{proof}
  Let $\mu$ be a measure supported on $E_C$. 
  We may assume that $\mu$ is doubling since otherwise $\udimreg(\mu) = \infty$. Observe that, by Theorem 4.1, it suffices to show that $\dima(E)< \udimreg(\mu)$. 

  By the COSC, $d = \dist(C,E)>0$. Fix $x \in C$ and choose $0<\roo<d$ such that $B(x,\roo) \subset U$, where $U$ is the open set from the COSC. Then we have for all $\iii \in \Sigma_*$ that $B(\fii_\iii(x),\roo \Lip(\fii_\iii) ) \subset \fii_\iii(U)$ and hence, by the COSC and the choices of $d$ and $\roo$,
  \begin{equation}\label{eq:roo-is-good}
    \mu(B(\fii_\iii(x),\roo \Lip(\fii_\iii))) \le \mu(\fii_\iii(C)) \le \mu(\fii_\iii(E_C)).
  \end{equation}
  Since $\fii_\iii(E_C) \subset B(\fii_\iii(x),\diam(\fii_\iii(E_C)))$ for all $\iii \in \Sigma_*$ and $\mu$ is doubling, 
  we thus obtain that
  \begin{align*}
    \mu(\fii_\iii(E_C)) &\le \mu(B(\fii_\iii(x),\diam(\fii_\iii(E_C)))) \\
    &\le D\mu(B(\fii_\iii(x),\roo\diam(E_C)^{-1} \diam(\fii_\iii(E_C)))) \\
    &= D\mu(B(\fii_\iii(x),\roo\Lip(\fii_\iii))) \le D\mu(\fii_\iii(C))
  \end{align*}
  for some constant $D>1$, independent of $\iii \in \Sigma_*$. Then
  \begin{align*}
    \mu(\fii_\iii(E_C)) &= \mu(\fii_\iii(C)) + \sum_{j \in \Sigma_1} \mu(\fii_{\iii j}(E_C)) \ge D^{-1}\mu(\fii_\iii(E_C)) + \sum_{j \in \Sigma_1} \mu(\fii_{\iii j}(E_C)),
  \end{align*}
  and so we have
  \begin{equation}\label{eq:am-muu}
    \sum_{j \in \Sigma_1} \mu(\fii_{\iii j}(E_C)) \le C_0\mu(\fii_\iii(E_C))
  \end{equation}
  for all $\iii \in \Sigma_*$, where $C_0=(1-D^{-1})<1$. 

  On the other hand, the self-similar set $E$ satisfies the open set condition, and hence $1 = \sum_{j \in \Sigma_1} \Lip(\fii_j)^\lambda$, where $\lambda = \dima(E) = \diml(E)$. From this we obtain that 
  \begin{equation}\label{eq:lip-pii}
    \Lip(\fii_\iii)^\lambda = \sum_{j \in \Sigma_1} \Lip(\fii_{\iii j})^\lambda
  \end{equation}
  for all $\iii \in \Sigma_*$. Combining~\eqref{eq:am-muu} and~\eqref{eq:lip-pii}, we see that
  \begin{equation*}
    \sum_{j \in \Sigma_1} \frac{\mu(\fii_{\iii j}(E_C))}{\mu(\fii_\iii(E_C))} \le \sum_{j \in \Sigma_1} C_0 \biggl(\frac{\Lip(\fii_{\iii j})}{\Lip(\fii_\iii)}\biggr)^\lambda,
  \end{equation*}
  and so we find for each $\iii \in \Sigma_*$ some $j\in\Sigma_1$ for which
  \begin{equation}\label{eq:the-j}
    \frac{\mu(\fii_{\iii j}(E_C))}{\mu(\fii_\iii(E_C))} \le C_0\biggl(\frac{\Lip(\fii_{\iii j})}{\Lip(\fii_\iii)}\biggr)^\lambda.
  \end{equation}
  Iteration of~\eqref{eq:the-j}, together with~\eqref{eq:roo-is-good}, shows that there is $\iii \in \Sigma$ such that
  \begin{align*}
    \mu(B(x_n,\roo\Lip(\fii_{\iii|_n}))) 
    & \le \mu(\fii_{\iii|_n}(E_C)) \le C_0^n \Lip(\fii_{\iii|_n})^{\lambda} \mu(E_C) \\
    & \le C_0^n \Lip(\fii_{\iii|_n})^{\lambda} \mu(B(x_n,\diam(E_C)))
  \end{align*}
  for every $n \in \N$, where $x_n=\fii_{\iii|_n}(x)$.

  Let $\lalpha = \min_i \Lip(\fii_i)<1$ and choose $0 < \eps < \log C_0^{-1} / \log\lalpha^{-1}$.
  Then $C_0\lalpha^{-\eps}<1$ and $\lalpha^n\le \Lip(\fii_{\iii|_n})$, and so it follows from the previous estimate that
  \begin{align*}
    \frac{\mu(B(x_n,\roo\Lip(\fii_{\iii|_{n}})))}{\mu(B(x_n,\diam(E_C)))} & \le C_0^n \Lip(\fii_{\iii|_n})^{\lambda} 
     = C_0^n \lalpha^{-n\eps} \lalpha^{n\eps}\Lip(\fii_{\iii|_n})^{\lambda}\\ 
     & \le (C_0 \lalpha^{-\eps})^n \Lip(\fii_{\iii|_n})^{\lambda+\eps}
  \end{align*}
  for all $n\in\N$. Since $(C_0 \lalpha^{-\eps})^n\to 0$ as $n\to\infty$, we conclude that $\udimreg(\mu) \ge \lambda +\eps>\dima(E)$, which proves the claim.
\end{proof}

For the lower regularity dimension, the statement corresponding to Theorem~\ref{thm:upper_reg}  would be that if $\diml(C)>\diml(E)$, then $\diml(E_C) > \ldimreg(\mu)$ for all measures $\mu$ supported on $E_C$. This, however, is not the case, as the following example shows.

\begin{example}\label{ex:non-VK}
 Consider again the usual $\frac 1 3$-Cantor set $E$, as in Example~\ref{ex:no_separation}. Let $C=[4/10,6/10]$ be the condensation set and take $\mu$ to be the restriction of the Lebesque measure to $E_C$. As noted in Example~\ref{ex:no_separation}, each ball $B(x,r)$ with $x\in E_C$ and $0<r\le 1$ contains an interval of length $cr$, where $0<c<1$ is independent of $x$ and $r$, and thus
 $cr\le \mu(B(x,r))\le 2r$ for all $x\in E_C$ and $0<r\le 1$. Hence $\mu$ is $1$-regular, and so in particular $\ldimreg(\mu)=1=\diml(E_C)$. 
\end{example}

This example is a special case of the following proposition.

\begin{proposition}\label{prop:non-VK}
 Let $C \subset \R^d$ be a compact $s$-regular set and let $\{ \fii_i \}$ be a similitude IFS satisfying the COSC with condensation $C$. If $E$ is the associated self-similar set such that $\diml(E)<s$, then for $\mu=\HH^{s}|_{E_C}$ it holds that $\ldimreg(\mu)=s=\diml(E_C)$.
\end{proposition}

\begin{proof}
 Again, as in Example \ref{ex:non-VK}, $\mu(B(x,r))\ge cr^s$ is clear, since each ball contains a relatively large copy of $C$.
 To show that $\mu(B(x,r))\le Cr^s$ observe first that, similarly as in the proof of Theorem \ref{thm:dim of inhomog}, the ball $B(x,r)$ intersects only $K_1$ many copies of $E_C$ at the scale $r$ where $K_1$ does not depend on $x$ nor $r$. Since $\fii_\iii(C)$ is $s$-regular, we have $\mu(\fii_\iii(C)) \lesssim \Lip(\fii_\iii)^s$ for all $\iii \in \Sigma_*$, where the associated constants do not depend on $\iii$. Therefore, it suffices to show that $\mu(E_C) \le \mu(\bigcup_{\iii \in \Sigma_*} \fii_\iii(C)) \lesssim \sum_{\iii \in \Sigma_*} \Lip(\fii_\iii)^s < \infty$. By \cite[Proposition 4.1]{Falconer1988}, a real number $t$ satisfies $\sum_i \Lip(\fii_i)^t = 1$ if and only if $t = \inf\{ s : \sum_{\iii \in \Sigma_*} \Lip(\fii_\iii)^s < \infty \}$. Since $t = \diml(E)$ satisfies $\sum_i \Lip(\fii_i)^t = 1$, the assumption $\diml(E) < s$ finishes the proof.
\end{proof}

Proposition~\ref{prop:non-VK} shows that we cannot hope to find for the lower regularity dimension a statement corresponding to Theorem~\ref{thm:upper_reg} using inhomogeneous self-similar sets with separation. This is in a huge contrast to the upper regularity dimension case. We are left with the following open question.

\begin{question}
 Does there exist a finitely doubling uniformly perfect complete metric space $X$ such that $\ldimreg(\mu) < \diml(X)$ for all measures $\mu$ supported on $X$.
\end{question}


\end{document}